\documentclass[english,11pt]{amsart}
\usepackage[english]{babel}
\usepackage{amsmath,amssymb,enumerate,amsthm}
\usepackage[latin1]{inputenc}
\usepackage[T1]{fontenc}
\usepackage{cite}
\usepackage{latexsym}
\usepackage{setspace}

\newtheorem{lemma}{Lemma}[section]
\newtheorem{teorema}[lemma]{Theorem}
\newtheorem{prop}[lemma]{Proposition}

\newtheorem{rk}[lemma]{Remark}

\newcommand{\SL}{\mathrm{SL}(2,\mathbb{Z})}

\newcommand{\A}{\mathcal{A}}

\newcommand{\z}{\mathfrak{z}}

\def\H{\mathcal{H}}
\def\={\;=\;}
\def\.={\;\dot{=}\;}

\def\O{\mathcal{O}}

\begin{document}

\title[Meromorphic analogues of modular forms in Shitani's lift]{Meromorphic analogues of modular forms generating the kernel of Shintani's lift}
\author{Paloma Bengoechea}
\address{Department of Mathematics, University of York, York, YO10 5DD, United Kingdom}
\email{paloma.bengoechea@york.ac.uk}

\maketitle

\begin{abstract} We study the meromorphic modular forms defined as sums of $-k$ ($k\geq 2$) powers of integral quadratic polynomials with negative discriminant. These functions can be viewed as meromorphic analogues of the holomorphic modular forms defined in the same way with positive discriminant, first investigated by Zagier in connection with the Doi-Naganuma map and then by Kohnen and Zagier in connection with Shimura-Shintani lifts. We compute the Fourier coefficients of these meromorphic modular forms and we show that they split into the sum of a meromorphic modular form with computable algebraic Fourier coefficients and a holomorphic cusp form.
\end{abstract}

\section{Introduction}

\let\thefootnote\relax\footnote{2010 Mathematics Subject Classification: 11F03, 11F11, 11F37.\\Key words and phrases: modular forms, complex multiplication.}

Eisenstein series are defined as sums of $-k$ powers ($k>2$ an even integer) of all linear functions with integer coefficients and variable in the upper complex half plane and play a crucial role in the theory of modular forms. It seems natural to look at similar series where we sum over functions with integer coefficients and higher degree. Sums of this kind, taken over quadratic polynomials with fixed positive discriminant, have been introduced in \cite{Z75} in connection with the Doi-Naganuma correspondence between elliptic modular forms and Hilbert modular forms. More precisely, for each discriminant $D>0$, the sum
\begin{equation}\label{f}
f_{k,D}(z)\=\pi^{-k}\sum_{\substack{(a,b,c)\in\mathbb{Z}^3\\b^2-4ac=D\\a>0}}\dfrac{1}{(az^2+bz+c)^k}\qquad (k\geq 2)
\end{equation}
is a cusp form of weight $2k$ for $\SL$ which arose, in the case where $D$ is a fundamental discriminant, by considering the restriction to the diagonal $z_1=z_2$ of a family of Hilbert modular forms $w_m(z_1,z_2)$ ($m=0,1,2,\ldots$) of weight $k$ for the Hilbert modular group $\mathrm{SL}_2(\O)$, where $\O$ is the ring of integers of the real quadratic field with discriminant $D$. The functions $w_m(z_1,z_2)$ are the Fourier coefficients of the kernel function for the Doi-Naganuma correspondence. They are well defined for all positive discriminants $D$ and so is $f_{k,D}(z)$. The Fourier coefficients of $f_{k,D}(z)$ at infinity are calculated in \cite{Z75} in terms of a Bessel function. 

The functions $2\pi^k D^{k-1/2}f_{k,D}(z)$ reappeared in the kernel function for the Shimura and Shintani lifts between half-integral and integral weight cusp forms (\cite{KZ80}). 
They can be interpreted as theta lifts: they are obtained by integrating the $D$-th weight $k+\frac{1}{2}$ classical cuspidal Poincar� series against Shintani's theta function projected into Kohnen's plus space (with respect to Petersson's inner product). 

In \cite{BKV}, the authors generalized the functions $f_{k,D}$ to a natural family of local Maass forms of weight $2k$ which may also be viewed as theta lifts of the $D$-th Poincar� series generalized to the context of weak Maass forms (with respect to Borcherds's regularized version of the Petersson inner product that one can find in \cite{B98}).

In \cite{KZ84}, Kohnen and Zagier calculated the period polynomials of the functions $f_{k,D}(z)$ for $D>0$. Later, Bringmann, Kane and Kohnen introduced in \cite{BKK} the theory of locally harmonic weak Maass forms and they found again the explicit even period polynomials of $f_{k,D}$ as an application of their theory. 

Also the even parts of the Eichler integrals of $f_{k,D}$ for $D>0$ have been studied because of their link with Diophantine approximation, reduction of binary quadratic forms, special values of zeta functions and Dedekind sums (\cite{Z99}, \cite{B}).

The functions $f_{k,D}(z)$ with $D<0$ have been investigated in author's Ph.D. thesis \cite{Bt} and have become interesting because of the similar properties with the case $D>0$. In the case $D<0$ , they have a non-holomorphic part determined by the points of complex multiplication of discriminant $D$. In the next section we prove the convergence of these functions and we calculate their Fourier coefficients in terms of a modified Bessel function. In the third section, we decompose $f_{k,D}(z)$ with $D<0$ into a meromorphic part with computable algebraic Fourier coefficients and a holomorphic cuspidal part with, a priori, transcendent Fourier coefficients. We denote by $\Gamma$ the group $\SL$, by $H$ the Hilbert class field $\mathbb{Q}(j(\O_D),\sqrt{D})$ of $\mathbb{Q}(\sqrt{D})$, and by $S_{2k}^H(\Gamma)$ the $H$-vector space of cusp forms of weight $2k$ for $\Gamma$ with Fourier coefficients (in the expansion at infinity) in $H$. 
We denote by $H_D(X)$ the class polynomial of discriminant $D$ and by $h(D)$ the class number. Throughout the paper we also use the standard notation $j(z)$ and $\Delta(z)$ for the modular j-invariant and the modular discriminant respectively.
Writing
$$
\Phi_D(z)=H_D(j(z))\, \Delta(z)^{h(D)/w},
$$
where $w=1$ if $D\neq -3,-4$; $w=2$ if $D=-4$, and $w=3$ if $D=-3$, we have

\textbf{Theorem I} 
\textit{
The function $\Phi_D(z)^k f_{k,D}(z)$ belongs to the space
\begin{equation}\label{yep}
S^H_{\left(\frac{12h(D)}{w}+2\right)k}(\Gamma)\, +\,  S^\mathbb{C}_{2k}(\Gamma)\, \Phi_D(z)^k.
\end{equation} 
In this decomposition, the modular form in $S^H_{\left(\frac{12h(D)}{w}+2\right)k}(\Gamma)$ is computable. 
}

We use a theorem of Borcherds (\cite{B99}) to find linear combinations of Hecke operators which, by acting on the functions $f_{k,D}(z)$, give explicit integral polynomials on these functions without any transcendent part in decomposition \eqref{yep}. Indeed, given $\underline{\lambda}=\left\{\lambda_n\right\}^\infty_{n=1}\in\oplus^\infty_{n=1}\mathbb{Z}$, denoting by $T_n$ the $n$-th Hecke operator acting on the space of meromorphic modular forms of weight $2k$ for $\Gamma$, and writing
$$
\phi_{\underline{\lambda}}=\sum^\infty_{n=1}\lambda_n T_n,\qquad f_{k,D,\underline{\lambda}}(z)\=f_{k,D}|\phi_{\underline{\lambda}}(z),
$$
we have

\textbf{Theorem II}
\textit{
If there exists a weakly holomorphic modular form
$$
g_{\underline{\lambda}}=\sum^\infty_{n=1} \lambda_n q^{-n}+ O(1)\in M^!_{2-2k}(\Gamma),
$$
then the function $\Phi_D(z)^k f_{k,D,\underline{\lambda}}(z)$ belongs to $S^H_{\left(\frac{12h(D)}{w}+2\right)k}(\Gamma)$.
}

In the appendix, we illustrate the results of section 3 giving the explicit decompositions of the functions $f_{k,D}(z)$ and $f_{k,D,\underline{\lambda}}(z)$ (for a convenient $\underline{\lambda}$) when $D=-3$ and $k\in\left\{2,3,4,5,6,7\right\}$.

\section{Convergence and Fourier coefficients}

Given an integer $D<0$ congruent to 0 or 1 modulo 4, and an integer $k\geq 2$ we define, for $z\in\H$,
$$
f_{k,D}(z)\=\pi^{-k}\sum_{\substack{(a,b,c)\in\mathbb{Z}^3\\b^2-4ac=D\\a>0}}\dfrac{1}{(az^2+bz+c)^k}.
$$
The positivity condition on $a$ does not play any important role, we include it to avoid summing simultaneously over $[a,b,c]$ and $[-a,-b,-c]$. The factor $\pi^{-k}$ is just a normalization factor.

Given a $\Gamma$-equivalence class $\A$ of integral binary quadratic forms with discriminant $D<0$, we define
\begin{equation}\label{f clase}
f_{k,D,\A}(z)\=\pi^{-k}\sum_{\substack{[a,b,c]\in\A \\a>0}}\dfrac{1}{(az^2+bz+c)^k}\qquad(z\in\H,\, k\geq 2).
\end{equation}

\begin{prop} The sums $f_{k,D}(z)$ and $f_{k,D,\A}(z)$ converge absolutely and uniformly. They are meromorphic modular forms of weight $2k$ for $\Gamma$.
\end{prop}

\begin{proof} For $z=x+iy\in\mathcal{H}$, we have
\begin{align*}
\pi^k|f_{k,D}(z)|&\leq\sum_{\substack{b^2-4ac=D\\a>0}}\dfrac{1}{\Big|a\Big(z+\dfrac{b+\sqrt{D}}{2a}\Big)\Big(z+\dfrac{b-\sqrt{D}}{2a}\Big)\Big|^k}\\
&\\
&\leq\!\!\!\!\!\sum_{\substack{b^2-4ac=D\\a>0}}\dfrac{1}{|a|^k\mathrm{max}\Big(\Big|x+\dfrac{b}{2a}\Big|,\Big|y+\dfrac{\sqrt{|D|}}{2a}\Big|\Big)^k\mathrm{max}\Big(\Big|x+\dfrac{b}{2a}\Big|,\Big|y-\dfrac{\sqrt{|D|}}{2a}\Big|\Big)^k}.
\end{align*}

For $R>0$ we count the number $N(R)$ of elements $(a,b)$ that occur in the last sum such that
$$
R\leq |a|\; \mathrm{max}\Big(\Big|x+\dfrac{b}{2a}\Big|,\Big|y+\dfrac{\sqrt{|D|}}{2a}\Big|\Big)\; \mathrm{max}\Big(\Big|x+\dfrac{b}{2a}\Big|,\Big|y-\dfrac{\sqrt{|D|}}{2a}\Big|\Big)<2R.
$$
The inequality
$$
\Big|a\Big(y^2-\dfrac{|D|}{4a^2}\Big)\Big|<2R
$$
implies $a=O(R)$. For fixed $a$, the inequality
$$
\Big|a\Big(x+\dfrac{b}{2a}\Big)^2\Big|<2R
$$
implies $b=O(R^\frac{1}{2})$. Hence $N(R)=O(R^\frac{3}{2})$ and
$$
\pi^k\, |f_{k,D}(z)|\, \ll\, \sum_{n=0}^\infty\dfrac{N(2^n)}{2^{nk}}\, \ll\, \sum_{n=0}^\infty\dfrac{1}{2^{n(k-\frac{3}{2})}},
$$
where the implicit constant depends on $z$, but it is still bounded for $z$ in a compact set. The last sum converges for all $k>\frac{3}{2}$.
\\

The convergence of $f_{k,D,\A}(z)$, which is a partial sum of $f_{k,D}(z)$, is immediate, and since the sum is taken over a $\Gamma$-equivalence class, we obtain the modularity.
 
The modularity of $f_{k,D}(z)$ follows, on the one hand, from the invariance of the discriminant $D$ under the action of $\Gamma$; on the other hand, from the fact that $a$ is always translation invariant and, for $D<0$, the coefficients $a$ and $c$ have the same sign, so the sign of $a$ is also invariant under the transformation $\begin{pmatrix}0&1\\-1&0\end{pmatrix}$.
\end{proof}

\begin{prop}
For $z\in\mathcal{H}$ with $\Im(z)>\frac{\sqrt{|D|}}{2}$, the Fourier expansion of the function $f_{k,D}(z)$ is
$$
f_{k,D}(z)\=\pi^{-k}\, \sum^\infty_{r=1}c_r\, e^{2\pi irz},
$$
\begin{equation}\label{coef fourier}
c_r\=\dfrac{2^{k+\frac{1}{2}}\, \pi^{k+1}\, r^{k-\frac{1}{2}}}{|D|^{\frac{k}{2}-\frac{1}{4}}\, (k-1)!}\, \sum^\infty_{a=1}\, a^{-\frac{1}{2}}\, S_{a,D}(r)\, I_{k-\frac{1}{2}}\Big(\dfrac{\pi r\sqrt{|D|}}{a}\Big),
\end{equation}
where
\begin{equation}
S_{a,D}(r)\=\sum_{\substack{b\!\pmod{2a} \\b^2\equiv D\!\pmod{4a}}} e^{\pi irb/a}
\end{equation}
and $I_{k-\frac{1}{2}}$ is a modified Bessel function, related to the Bessel function of the first kind $J_{k-\frac{1}{2}}$ by $I_{k-\frac{1}{2}}(x)=i^{-k+\frac{1}{2}}\, J_{k-\frac{1}{2}}(ix)$.
\end{prop}

\begin{proof}
We can calculate the Fourier expansion of $f_{k,D}(z)$ by splitting the sum:
$$
\pi^kf_{k,D}(z)\=\sum^\infty_{a=1}f_{k,D}^a(z),
$$
where
$$
f_{k,D}^a(z)\=\sum_{\substack{b\in\mathbb{Z}\\b^2\equiv D\!\pmod{4a}}}\Big(az^2+bz+\dfrac{b^2+|D|}{4a}\Big)^{-k}.
$$
We can split the sum again:
$$
f_{k,D}^a(z)\=\sum_{\substack{b\!\pmod{2a}\\b^2\equiv D\!\pmod{4a}}}\sum_{n\in\mathbb{Z}}\Big(a(z+n)^2+b(z+n)+\dfrac{b^2+|D|}{4a}\Big)^{-k}.
$$
The first sum is finite and for each term in the sum, the $r$-th Fourier coefficient is given by
\begin{equation}\label{int coef fourier}
\int^{\infty+iC}_{-\infty+iC}\Big(az^2+bz+\dfrac{b^2+|D|}{4a}\Big)^{-k}e^{-2\pi irz}dz\qquad (C>\dfrac{\sqrt{|D|}}{2}).
\end{equation}
This integral vanishes for $r\leq 0$ and we can calculate it for $r>0$ changing variables $z=it-\frac{b}{2a}$:
\begin{align*}
\int^{\infty+iC}_{-\infty+iC}\dfrac{e^{-2\pi irz}}{\Big(az^2+bz+\dfrac{b^2+|D|}{4a}\Big)^k}dz \=-&\dfrac{i\, e^{\pi irb/a}}{a^k}\int^{C+i\infty}_{C-i\infty}\dfrac{e^{2\pi rt}}{\Big(t^2-\dfrac{|D|}{4a^2}\Big)^k}dt\\
\\
\= \dfrac{2^{k+\frac{1}{2}}\, \pi^{k+1}\, r^{k-\frac{1}{2}}}{|D|^{\frac{k}{2}-\frac{1}{4}}\, \sqrt{a}\, (k-1)!}\, e^{\pi irb/a}\,  I_{k-\frac{1}{2}}\Big(\dfrac{\pi r\sqrt{|D|}}{a}\Big).
\end{align*}
(The last integral is calculated in \cite{A}, (29.3.60)).
\end{proof}

\section{Algebraic properties of Fourier coefficients}

Before we state and prove the main theorems of this section, we introduce the notation and some results in the theory of modular forms that we need. All the results mentioned here can be found in \cite{Z08}, in the first part, written by D. Zagier.
For a discriminant $D<0$ we denote by $H_D(X)$ the class polynomial
$$
H_D(X)\=\prod_{\mathfrak{z}\in\Gamma\backslash\mathfrak{Z}_D}(X-j(\mathfrak{z}))^{1/w},
$$
where $\mathfrak{Z}_D\subset\H$ is the set of CM points of discriminant $D$ and 
$$
w=\left\{\begin{array}{ll}3 \quad &\mbox{if $D=-3$}\\
2 &\mbox{if $D=-4$}\\
1 &\mbox{if $D\neq -3,-4$}.
\end{array}\right.
$$
We write
\begin{equation}
H_D(j(z))\=\dfrac{\Phi_D(z)}{\Delta(z)^{h(D)/w}}
\end{equation}
with $\Phi_D(z)\in M_{\frac{12h(D)}{w}}(\Gamma)$.

For a $\Gamma$-equivalence class $\A$ of integral binary quadratic forms of discriminant $D<0$, we denote by $\z$ 
a CM point of $\A$ and we define the modular form $\Phi_\A(z)$ of weight $\frac{12}{w}$ by
$$
\Phi_\A(z)\=((j(z)-j(\z))\Delta(z))^{1/w}\=\left\{\begin{array}{ll} 
E_4(z) &\mbox{if $D=-3$}\\
E_6(z) &\mbox{if $D=-4$}\\
E_4(z)^3-j(\z)\, \Delta(z) &\mbox{if $D\neq -3,-4$}.
\end{array}\right.
$$

The poles in $\mathcal{H}$ of the functions $f_{k,D,\A}(z)$ and $f_{k,D}(z)$ are the respective sets $\Gamma \z$ and $\mathfrak{Z}_D$. But $\Gamma \z$ and $\mathfrak{Z}_D$ are also the respective sets of zeros of the functions $\Phi_\A(z)$ and $\Phi_D(z)$. Hence the functions $\Phi_{\A}(z)^k f_{k,D,\A}(z)$ and $\Phi_D(z)^k f_{k,D}(z)$ are cusp forms of respective weight $(\frac{12}{w}+2)k$ and $(\frac{12h(D)}{w}+2)k$.

Since $\Phi_{\A}(z)^k f_{k,D,\A}(z)$ and $\Phi_D(z)^k f_{k,D}(z)$ are holomorphic on the upper-half plane, they have a Taylor expansion in a neighborhood of each  point $z=x+iy\in\H$. The Taylor expansion in the classic sense is not the natural expansion for modular forms because it only converges on a disk centered at $z$ and tangent to the real line, whereas the domain of holomorphy of a modular form is the whole upper-half  plane $\H$. Recall that, for fixed $z\in\H$, the map sending $z'\in\H$ to $w=\dfrac{z'-z}{z'-\bar{z}}$ is an isomorphism between $\H$ and the open unit disk centered at $z$ (the inverse map sends $w$ to $z'=\dfrac{z-\bar{z}w}{1-w}$). It is then more natural to expand a modular form under the action of the transformation $\begin{pmatrix}-\bar{z} &z\\ -1 &1\end{pmatrix}\in\mathrm{SL}(2,\mathbb{C})$ in power series of $w$ (see \eqref{taylor} below). Throughout the paper, we call this new expansion \textit{modified Taylor expansion}.

Another classical concept which is usually modified in the theory of modular forms is the concept of derivative. The classical derivative of a modular form not being modular, several differentiation operators have been introduced. We consider the operator defined, for a modular form $f(z)$ of weight $2k$, by
\begin{equation}\label{operador}
\partial f(z)\=\dfrac{1}{2\pi i}\, \dfrac{df}{dz}(z)-\dfrac{2k}{4\pi y}\, f(z),
\end{equation}
where $y=\Im(z)$. The function $\partial f(z)$ is not holomorphic anymore, but it is modular of weight $2k+2$. Moreover,
if $K$ is an imaginary quadratic field, $z\in K\cap\H$ is a CM point, and the Fourier coefficients of $f(z)$ belong to an algebraic field $L$, then, for all $n\geq 0$, the value $\partial^nf(z)$ is an algebraic multiple (the constant of multiplication belongs to $L$) of the $(2k+2n)$-th power of the Chowla-Selberg period $\Omega_K$ defined by (see \cite{Z08}, the corollary in �6.3)
$$
\Omega_K\=\dfrac{1}{\sqrt{2\pi|D|}}\, \left(\prod^{|D|-1}_{m=1} \Gamma\Big(\dfrac{m}{|D|}\Big)^{\chi_D(m)}\right)^{w/2h(D)}. 
$$
Finally, the natural expansion of a modular form $f(z)$ of weight $2k$ in a neighborhood of a point $z=x+iy\in\H$ in terms of the operator \eqref{operador} is given explicitly by (Prop. 17 on p. 52 in \cite{Z08})
\begin{equation}\label{taylor}
(1-w)^{-2k}\, f\Big(\dfrac{z-\bar{z}w}{1-w}\Big)\=\sum_{n=0}^{\infty} \partial^nf(z)\, \dfrac{(4\pi yw)^n}{n!}\qquad(|w|<1).
\end{equation}
The main interest of this expansion is that, after normalization dividing by suitable powers of the period $\Omega_K$, all the coefficients are algebraic and belong to the field of definition of the Fourier coefficients of $f(z)$.

Below we keep the notation $H=\mathbb{Q}(j(\O_D),\sqrt{D})$ of the introduction.

\begin{teorema} \label{expr f}
For $k\geq 2$, we have
\begin{equation}\label{primera}
\Phi_\A(z)^k f_{k,D,\A}(z) \, \in\, S^H_{\left(\frac{12}{w}+2\right)k}(\Gamma)\, +\,  S^\mathbb{C}_{2k}(\Gamma)\, \Phi_\A(z)^k,
\end{equation} 
\begin{equation}\label{segunda}
\Phi_D(z)^k f_{k,D}(z)\, \in\, S^H_{\left(\frac{12h(D)}{w}+2\right)k}(\Gamma)\, +\,  S^\mathbb{C}_{2k}(\Gamma)\, \Phi_D(z)^k,
\end{equation}
In these decompositions, the modular forms in $S^H_{\left(\frac{12}{w}+2\right)k}(\Gamma)$ and $S^H_{\left(\frac{12h(D)}{w}+2\right)k}(\Gamma)$ are computable. 
\end{teorema}

\begin{proof}  Since the proofs of \eqref{primera} and \eqref{segunda} are analogous, we focus on the proof of \eqref{primera}. The value $w$ does not play any role in the proof either, so we assume $w=1$ to avoid tedious notation.

In the usual fundamental domain for the action of $\Gamma$ on the upper-half plane, the function $f_{k,D,\A}(z)$ has only one pole $\z=x+iy$, which is the root in $\mathcal{H}$ of some reduced binary quadratic form $[a,b,c]\in\A$. By considering the definition \eqref{f clase} we have, in a neighborhood of $\z$,
\begin{equation}\label{taylor f}
f_{k,D,\A}(z)\=\dfrac{1}{(\pi a(z-\z)(z-\overline{\z}))^k}\, +\, O(1).
\end{equation}
Replacing $f_{k,D,\A}$ in \eqref{taylor f} by its image by the transformation $\begin{pmatrix}-\z' &\z\\-1 &1\end{pmatrix}$, we have
\begin{equation}\label{1}
(1-w)^{-2k}f_{k,D,\A}\Big(\dfrac{\z-\overline{\z} w}{1-w}\Big)\=\dfrac{w^{-k}}{\pi^k\, (2aiy)^{2k}}\, +\, O(1)\qquad (|w|<1).
\end{equation}
Since $\z$ is a zero of order $k$ of $\Phi_\A^k(z)$, the first non-zero coefficient in the modified Taylor expansion of $\Phi_\A^k(z)$ is the coefficient of $w^k$: 
\begin{equation}\label{aux1}
(1-w)^{-12k}\Phi_\A^k\Big(\dfrac{\z-\overline{\z} w}{1-w}\Big)\=\sum^{\infty}_{n=k}\partial^{n}\Phi_\A^k(\z)\, \dfrac{(4\pi yw)^n}{n!}.
\end{equation}
Combining \eqref{1} with \eqref{aux1}, we compute the expansion of $\Phi_\A^k\, f_{k,D,\A}$ up to the $(k-1)$-st power (included):
\begin{equation}\label{taylor expa}
(1-w)^{-14k} \Phi_\A^k\, f_{k,D,\A}\Big(\dfrac{\z-\overline{\z} w}{1-w}\Big)\=\sum^{k-1}_{n=0}\partial^{k+n}\Phi_\A^k(\z)\, \dfrac{(4\pi yw)^n}{(n+k)!\, a^{2k}\, (-y)^k}+O(w^k).
 \end{equation}

Since the function $\Phi_\A(z)^k\, f_{k,D,\A}(z)$ belongs to the space $S_{14k}(\Gamma)$, it can be written as
\begin{equation}\label{formula}
\Phi_\A(z)^k\, f_{k,D,\A}(z)=\ast\, E_4(z)^\delta\, E_6(z)^\epsilon \Delta(z)^M,
\end{equation}
where $\ast$ is a complex constant and
$$
14k=4\delta+6\epsilon+12M,
$$
with $\delta\in\left\{0,1,2\right\}$, $\epsilon\in\left\{0,1\right\}$, $M=\mathrm{dim}(S_{14k}(\Gamma))=\left[\frac{14k}{12}\right]$. Moreover, the triple $(\delta,\epsilon,M)$ and the constant $\ast$ are uniquely determined.

The functions
$$
\Psi_{m,\A}(z)\=(j(z)-j(\z))^m\, E_4(z)^\delta\, E_6(z)^\epsilon\, \Delta(z)^{M} \qquad (0\leq m\leq M-1)
$$
are a basis for the space $S_{14k}(\Gamma)$ (as a $\mathbb{C}$-vector space).
We can write $\Phi_\A(z)^k\, f_{k,D,\A}(z)$ in this basis:
\begin{equation}\label{2}
\Phi_\A(z)^k\, f_{k,D,\A}(z)\=\sum_{m=0}^{M-1}c_m\, \Psi_{m,\A}(z)
\end{equation}
with $c_m\in\mathbb{C}$ for $m=0,\ldots,M-1$.
The functions $\dfrac{\Psi_{m,\A}(z)}{\Phi^k_{\A}(z)}$ are cusp forms of weight $2k$ for $m=k,\ldots,M-1$. We need to compute the coefficients $c_m$ for $0\leq m\leq k-1$ and show that they belong to $H$. We can expand again $\Phi_\A(z)^k\, f_{k,D,\A}(z)$, using now the expression \eqref{2}, in terms of the coefficients $c_m$ and the modified Taylor expansion of the functions $\Psi_{m,\A}$ in a neighborhood of $\z$ (recall that $\z$ is a zero of order $m$ of $\Psi_{m,\A}$ and so the first non-zero coefficient in the expansion of $\Psi_{m,\A}$ is the coefficient of the term of order $m$):
\begin{equation}\label{aux2}
(1-w)^{-14k}\Phi_\A(z)^k\, f_{k,D,\A}\Big(\dfrac{\z-\overline{\z} w}{1-w}\Big)\=\sum_{m=0}^{M-1}c_m\sum^\infty_{n=m}\partial^n\Psi_{m,\A}(\z)\, \dfrac{(4\pi yw)^n}{n!}.
\end{equation} 
By comparing the expansions \eqref{taylor expa}, and \eqref{aux2} and by matching the coefficients for each power of $w$ from both expressions, we make explicit the constants $c_m$ for $m=0,\ldots, k-1$:
$$
c_0\=\dfrac{\partial^k\Phi^k_\A(\z)}{k!\, a^{2k}\, (-y)^k\, \Psi_{0,\A}(\z)},
$$
$$
c_m\=\dfrac{m!\, \partial^{m+k}\Phi_{\A}^k(\z)}
{(m+k)!\, a^{2k}\, (-y)^k\, \partial^{m}\Psi_{m,\A}(\z)}-\dfrac{\sum^{m-1}_{n=0}c_n\, \partial^{m}\Psi_{n,\A}(\z)}{\partial^{m}\Psi_{m,\A}(\z)}.
$$

For $j,l\geq 0$, the values $\partial^{j}\, \Psi_{l,\A}(\z)$ and $\partial^{j+k}\Phi_\A(\z)^k$ belong to $\mathbb{Q}(j(\z),\z)\, \Omega_D^{14k+2j}$, so the coefficients $c_m$ belong to the field $\mathbb{Q}(j(\z),\z)$ for all $m=0,\ldots, k-1$. 
\end{proof}

We call transcendent parts the projections on the second space with complex coefficients in the decompositions \eqref{primera} and \eqref{segunda}. The transcendent parts are cuspidal but unfortunately we cannot compute them. We show below how we use a theorem of Borcherds to find certain explicit integral polynomials in the functions $f_{k,D}$ ($k$ is fixed and $D$ takes negative values) which get rid of the transcendent parts. These polynomials are given by suitable linear combinations of Hecke operators acting on $f_{k,D}(z)$.
For each prime $p$, the action of the $p$-th Hecke operator $T_p$ on $f_{k,D}$ is given in \cite{KZ80} by the closed form
\begin{equation}\label{fkd hecke primo}
f_{k,D}|_{2k}T_p\=p^{2k-1}f_{k,Dp^2}+\left(\dfrac{D}{p}\right)p^{k-1}f_{k,D}+f_{k,\frac{D}{p^2}},
\end{equation}
with the convention $f_{k,\frac{D}{p^2}}=0$ if $p^2\nmid D$ (the proof of this relation is exactly the same as the proof of the analogous statement for non-holomorphic modular forms of weight zero given in \cite{Z81} (proof of equation (36))).
Since the Hecke algebra is generated in an explicit way by the $p$-th Hecke operators through the identities
$$
\begin{array}{llll}
T_{mn}&=&T_m T_n &\qquad\mbox{if $(m,n)=1$},\\
T_{p^{r+1}}&=&T_{p^r}T_p-p^{k-1}T_{p^{r-1}} &\qquad\mbox{if $r\geq 1$ and $p$ is prime},
\end{array}
$$
the action of the $n$-th ($n$ not necessarily prime) Hecke operator $T_n$ on $f_{k,D}$ can also be written in a closed form, as an explicit integral polynomial in the functions $f_{k,D}$. 

Given $\underline{\lambda}=\left\{\lambda_n\right\}^\infty_{n=1}\in\oplus^\infty_{n=1}\mathbb{Z}$ and $k\geq 2$, we write 
$$
\phi_{\underline{\lambda}}=\sum^\infty_{n=1}\lambda_n T_n,\qquad\quad
f_{k,D,\underline{\lambda}}(z)\=f_{k,D}|\phi_{\underline{\lambda}}(z).
$$

\begin{teorema}\label{teo final} 
If there exists a weakly holomorphic modular form
\begin{equation}\label{g}
g_{\underline{\lambda}}=\sum^\infty_{n=1} \lambda_n q^{-n}+ O(1)\in M^!_{2-2k}(\Gamma),
\end{equation}
then the function $\Phi_D(z)^k f_{k,D,\underline{\lambda}}(z)$ belongs to the space $S^H_{\left(\frac{12h(D)}{w}+2\right)k}(\Gamma)$.

(Note that $\lambda_n=0$ $\forall n\geq N$ for some $N\geq 0$, so $g_{\underline{\lambda}}$ has only finitely many non-vanishing terms in its principal part.)
\end{teorema}

\begin{proof} According to Theorem \ref{expr f}, the function $\Phi_D(z)^k f_{k,D,\underline{\lambda}}(z)$ belongs to the space
\begin{equation}\label{descomposicion}
S^H_{\left(\frac{12h(D)}{w}+2\right)k}(\Gamma)\, +\,  S^\mathbb{C}_{2k}(\Gamma)\, \Phi_D(z)^k.
\end{equation} 
We have to show that there is no transcendent part in the decomposition \eqref{descomposicion}  if there exists a weakly holomorphic modular form $g_{\underline{\lambda}}$ of the form \eqref{g}.

If there exists a weakly holomorphic modular form such as \ref{g}, then, by \cite{B99} (section 3),
\begin{equation}\label{prop M}
\sum^\infty_{n=1}\lambda_n c_n=0 \quad\mbox{ for any }
f=\sum^\infty_{n=1}c_n q^n\in S_{2k}(\Gamma).
\end{equation}
If $f=\sum^\infty_{n=1}c_nq^n$ is a cusp form of weight $2k$ for $\Gamma$, so is $T_m(f)$, which Fourier expansion is (see \cite{Z08}, �4.1 of the first part)
\begin{equation}\label{Taylor}
T_m(f)=\sum^\infty_{n=1} \Big(\sum_{\substack{a\mid n\\a\mid m}}a^{k-1}c_{nm/a^2}\Big) q^n.
\end{equation}
Hence  \eqref{prop M} implies 
\begin{equation}\label{coef}
\sum_{n=1}^\infty \lambda_n\Big(\sum_{\substack{a\mid n\\a\mid m}}a^{k-1}c_{nm/a^2}\Big)=0.
\end{equation}
But the expression in \eqref{coef} is the $m$-th coefficient in the Fourier expansion of $f|\phi_{\underline{\lambda}}$, so $f|\phi_{\underline{\lambda}}=0$. In particular, if we denote by $f(z)$ the transcendent part in the decomposition
$$f_{k,D}(z)\, \in\, \dfrac{S^H_{(\frac{12}{w}+2)k}(\Gamma)}{\Phi_D(z)^k}\, +\,  S^\mathbb{C}_{2k}(\Gamma),$$
we have $f|\phi_{\underline{\lambda}}=0$. 
\end{proof}

\begin{rk} In the proof of Theorem \ref{teo final}, we deduce the statement (i) below from the statement (ii):
\begin{enumerate}[(i)]
\item $f|\phi_{\underline{\lambda}}=0$ for any cusp form $f\in S_{2k}(\Gamma)$,

\item there exists a weakly holomorphic modular form
$$
g_{\underline{\lambda}}=\sum^\infty_{n=1} \lambda_n q^{-n}+ O(1)\in M^!_{2-2k}(\Gamma).
$$
\end{enumerate}
In fact these two statements are equivalent. On the one hand, the statement (ii) and the condition \eqref{prop M} are equivalent (see \cite{B99}). On the other hand, the function  
 $f|\phi_{\underline{\lambda}}=\sum_{n=1}^\infty \lambda_n T_n(f)$ in statement (i), defined from a cusp form $f=\sum_{n=1}^\infty c_n q^n$, is identically zero if and only if 
\begin{equation}\label{coef Fourier}
\sum^\infty_{n=1}\lambda_n\Big(\sum_{\substack{a\mid n\\a\mid m}}a^{k-1}c_{nm/a^2}\Big)=0\qquad\forall m\geq 1.
\end{equation}
For $m=1$ \eqref{coef Fourier} becomes 
$$
\sum^\infty_{n=1}\lambda_n c_n=0.
$$ 
\end{rk}

\section{Appendix: Explicit examples}

In this section we give the explicit decomposition \eqref{segunda} for $D=-3$ and $k\in\left\{2,3,4,5,6,7\right\}$. We also give a sequence $\underline{\lambda}$ such that, for any $k$ in the previous set, $f_{k,-3,\underline{\lambda}}(z)$ has no transcendent part in the decomposition \eqref{descomposicion}.

There is only one $\Gamma$-equivalence class of binary quadratic forms of discriminant -3; we denote by $[1,1,1]$ the reduced quadratic form in this class and by $z_{-3}=\dfrac{1+\sqrt{3}i}{2}$ its root in $\H$. The Hilbert extension of the quadratic field $\mathbb{Q}(\sqrt{3}i)$ is trivial. For $k\in\left\{2,3,4,5,7\right\}$, the space $S_{2k}(\Gamma)$ is trivial, so there is no transcendent part in the decomposition \eqref{segunda} of $f_{k,-3}$. The value $k=6$ is more interesting, so we give the details of the calculations for $f_{6,-3}(z)$ and we list the calculated expressions of $f_{k,-3}(z)$ in a table for $k\in\left\{2,3,4,5,7\right\}$.

For \textbf{$k=6$}, the space $S_{36}(\Gamma)$ is generated by $\Delta^3(z),\Delta^2(z)E_4^3(z), \Delta(z) E_4^6(z)$. There are three complex constants $C_0, C_1, C_2$ such that
\begin{equation}\label{6,-3}
E_4^6(z)f_{6,-3}(z)=C_0\Delta(z)^3+C_1\Delta(z)^2E_4(z)^3+C_2\Delta(z) E_4(z)^6.
\end{equation}
In order to compute $C_0$ and $C_1$ we compare the modified Taylor expansions in a neighborhood of $z_{-3}$ from both sides of the equality \eqref{6,-3}. 
We have

\begin{align*}
&(1-w)^{-36} \Delta^3\Big(\dfrac{z_{-3}-\overline{z_{-3}}w}{1-w}\Big) = -\; \Omega_{-3}^{36} - 2^3 \cdot3\; \pi^3\; \Omega_{-3}^{42}\;  w^3 + O(w^6);
\\
&(1-w)^{-36} \Delta^2E_4^3\Big(\dfrac{z_{-3}-\overline{z_{-3}}w}{1-w}\Big) = - 2^{12}\cdot3^3\; \pi^3\; \Omega_{-3}^{42}\; w^3 +O(w^6);
\end{align*}
\begin{align*}
&(1-w)^{-36} \Delta E_4^6\Big(\dfrac{z_{-3}-\overline{z_{-3}}w}{1-w}\Big) = O(w^6);
\\
&(1-w)^{-24} E_4^6\Big(\dfrac{z_{-3}-\overline{z_{-3}}w}{1-w}\Big) = 2^{24}\cdot3^6\; \pi^6\; \Omega_{-3}^{36}\; w^6 - 2^{25}\cdot3^7\cdot 5\; \pi^9\; \Omega_{-3}^{42}\; w^9 + O(w^{12});
\\
&(1-w)^{-12} f_{6,-3}\Big(\dfrac{z_{-3}-\overline{z_{-3}}w}{1-w}\Big) = \dfrac{w^{-6}}{\pi^6\; 3^6} + O(1).
\end{align*}
The equation \eqref{6,-3} becomes
$$
2^{24}\Omega_{-3}^{36}  - 2^{25}\cdot3\cdot5\pi^3\Omega_{-3}^{42}w^3 + O(w^6) = -C_0\Omega_{-3}^{36} - (2^3 \cdot 3C_0 + 2^{12}\cdot3^{3}C_1)\pi^3 \Omega_{-3}^{42}w^3 + O(w^6).
$$

If we compare the constant terms and the cubic powers of $w$, we find
$$
C_0 = -2^{24}, \quad\qquad C_1= 2^{13}.
$$
Hence
$$
f_{6,-3}=\dfrac{-2^{24}\Delta^3+2^{13}\Delta^2E_4^3}{E_4^6}+C_2\Delta,
$$
where $C_2$ is the coefficient of $q$ in the Fourier expansion of $f_{6,-3}(z)$. It is given by \eqref{coef fourier} and has the  numerical value $C_2\approx -550.5139$.

For \textbf{$k\in\left\{2,3,4,5,6,7\right\}$} the expressions of $f_{k,-3}$ are listed in the table below:
\begin{center}
\begin{tabular}{c|c}
 $k$ &$f_{k,-3}$ \\
\hline
&\\
2 &$-2^8\dfrac{\Delta}{E_4^2}$\\
&\\
3 &$-\dfrac{2^9 E_6\, \Delta}{3\sqrt{3}\, E_4^3}$\\
&\\
4 &$\dfrac{2^{16}\cdot3^2\, \Delta^2-2^6\Delta\, E_4^3}{3\, E_4^4}$\\
&\\
5 &$\dfrac{2^{17}\cdot3^3 E_6\, \Delta^2 - 2^5\cdot13\, E_4^3\, E_6\, \Delta}{3^4\sqrt{3}\, E_4^5}$\\
&\\
6 &$\dfrac{-2^{24}\Delta^3+2^{13}\Delta^2E_4^3}{E_4^6}+C\Delta$,\, 
where $C\approx -550.5139$\\
&\\
7 &$\dfrac{-2^{25}\cdot3^5\cdot 5\, E_6\, \Delta^3+2^{13}\cdot3^2\cdot31\, E_4^3\, E_6\, \Delta^2 -2^7E_4^6\, E_6\, \Delta}{3^6\cdot5\sqrt{3}\, E_4^7}$\\
\end{tabular}
\end{center}

Note that if $k\in\left\{2,3,4,5,7\right\}$, then
$$
\dfrac{E_{14-2k}(z)}{\Delta(z)}=q^{-1}+O(1)\in M_{2-2k}^!(\Gamma).
$$ 
If we choose $\underline{\lambda}=(1,0,0,\ldots)$, then $f_{k,-3,\underline{\lambda}}=f_{k,-3}$ and, according to Theorem \ref{teo final}, there is no transcendent part in the decomposition \eqref{segunda} (this agrees with the triviality of $S_{2k}(\Gamma)$).
\\

For $k=6$, since
$$
\dfrac{E_4(z)^2E_6(z)}{\Delta(z)^2}=q^{-2}+24q^{-1}+O(1)\in M_{-10}^!(\Gamma),
$$
if we choose $\underline{\lambda}=(24,1,0,0,\ldots)$, then
\begin{align*}
f_{6,-3,\underline{\lambda}}&=24\, f_{6,-3}|T_2+f_{6,-3}\\
&=24\, \Big(2^{11}f_{6,-12}+2^5\left(\dfrac{-3}{2}\right)f_{6,-3}\Big)+f_{6,-3}.
\end{align*}
The functions $f_{6,-3}$ and $f_{6,-12}$ have both transcendent cuspidal parts:
$$
f_{6,-12}=\dfrac{-2^{12}\cdot3^2\Delta^3 + 2\cdot41\, \Delta^2E_4^3}{3^2E_4^6}+C'\Delta,
$$
but the new combination $f_{6,-3,\underline{\lambda}}$ is again purely algebraic:
$$
f_{6,-3,\underline{\lambda}}=\dfrac{-2^{24}\cdot 3\cdot 13\, \Delta^3+2^{13}\cdot167\, \Delta^2E_4^3}{3E_4^6}.
$$

\textbf{Acknowledgements.} This work is part of my PhD thesis. I wish to express my gratitude to Don Zagier for his precious advice in discussing mathematics in his supervising. I would like to thank Pilar Bayer for her careful reading of this paper.
 

\end{document}